\documentclass[review]{elsarticle}
\date{}
\usepackage{amsmath,amsthm,amsfonts,latexsym,amssymb,fancybox,comment,graphicx,subfigure,epsfig,eepic,epic,epstopdf,array,subeqnarray}
\usepackage{calrsfs,mathrsfs,amsbsy}
\DeclareMathAlphabet{\pazocal}{OMS}{zplm}{m}{n}
\usepackage{datetime}
\journal{Journal of Computational and Applied Mathematics}
\usepackage{colortbl,color}
\usepackage[colorlinks,linktocpage,linkcolor=blue]{hyperref}
\usepackage{hypernat,mathrsfs}

\usepackage{colortbl}

\usepackage[all]{xy}

\numberwithin{equation}{section}

\newtheorem{theorem}{Theorem}[section]
\newtheorem{lemma}[theorem]{Lemma}

\newtheorem{remark}[theorem]{Remark}
\newtheorem{corollary}[theorem]{Corollary}

\newcommand{\thmref}[1]{Theorem~\ref{#1}}
\newcommand{\thmrefs}[2]{Theorems~\ref{#1} and ~\ref{#2}}
\newcommand{\secref}[1]{\S\ref{#1}}
\newcommand{\lemref}[1]{Lemma~\ref{#1}}
\newcommand{\lemrefs}[2]{Lemmas~\ref{#1} and ~\ref{#2}}

\newcommand{\figref}[1]{Figure~\ref{#1}}

\newcommand{\tabref}[1]{Table~\ref{#1}}
\newcommand{\tabrefs}[2]{Tables~\ref{#1} and ~\ref{#2}}
\newcommand{\eq}[1]{\begin{eqnarray}\label{#1}}
\newcommand{\qe}{\end{eqnarray}}

\newcommand{\nn}{\nonumber}

\newcommand{\be}{\begin{eqnarray}}
\newcommand{\ee}{\end{eqnarray}}
\newcommand{\bal}{\begin{aligned}}
\newcommand{\eal}{\end{aligned}}
\newcommand{\bes}{\begin{eqnarray*}}
\newcommand{\ees}{\end{eqnarray*}}
\newcommand{\bs}{\begin{subeqnarray}}
\newcommand{\es}{\end{subeqnarray}}
\newcommand{\bss}{\begin{subeqnarray*}}
\newcommand{\ess}{\end{subeqnarray*}}

\newcounter{saveeqn}

\def\hat{\widehat}
\def\hQ{\widehat Q}

\def\supp{\operatorname{supp}}
\def\Span{\operatorname{Span}}
\def\Oh{\mathcal O}

\def\p{\partial}

\def\dsum{\displaystyle\sum}

\def\O{\Omega}
\def\cE{\mathcal E}

\def\DSSY{DSSY}
\def\vDSSY{\mathbf{DSSY}}

\def\vDSSYhz{\mathbf{DSSY}^h_0}
\def\vcQhz{\pmb{\mathscr{Q}}_{1,0}^{c,h}} 
\def\NChz{{\mathscr{P}}_{1,0}^{nc,h}}
\def\vNCh{\pmb{\mathscr{P}}_1^{nc,h}}
\def\vNChz{\pmb{\mathscr{P}}_{1,0}^{nc,h}}
\def\ttvNChz{\widetilde{\widetilde{\vNChz}}}
\def\msvNChz{\pmb{\mathscr{P}}_{1,0}^{me,2h}} 

\def\bq{\overline{q}_h}
\def\gcbs{{\mathcal C}^h}   
\def\gcb{{\mathbb C}_h}   
\def\gbbs{{\pmb{\mathscr B}}^h}   
\def\gbb{{\mathbb B}_h}   
\def\cdfs{{\pmb{\mathscr D}}}   
\def\gdfs{{\pmb{\mathscr D}}^h}   
\def\ttgdfs{\widetilde{\widetilde{\pmb{\mathscr D}^h}}}   
\def\Pch{{\mathcal P}_{c}^h}
\def\Phz{{\mathscr{P}}_0^{h}}
\def\Phhz{{\mathscr{P}}_0^{2h}}
\def\Pcf{\widetilde{\Phz}}

\def\De{\Delta}

\def\and{\quad\text{and}\quad}

\def\<{\left\langle}
\def\>{\right\rangle}

\def\bL{\mathbf L}
\def\bP{\mathbf P}
\def\bQ{\mathbf Q}

\def\bH{\mathbf H}
\def\bV{\mathbf V}
\def\bZ{\mathbf Z}

\def\bu{\mathbf u}
\def\bv{\mathbf v}
\def\bw{\mathbf w}
\def\b1{\mathbf 1}
\def\bx{\mathbf x}
\def\bb{\mathbf b}

\def\bn{\mathbf n}

\def\Tau{{\mathcal T}}

\def\bnu{{\boldsymbol \nu}}

\def\bpsi{{\boldsymbol \psi}}
\def\bPi{{\boldsymbol \Pi}}

\def\grad{\nabla\,}
\def\div{\nabla\cdot}

\def\dim{\operatorname{dim}\,}

\def\XXint#1#2#3{{\setbox0=\hbox{$#1{#2#3}{\int}$ }
\vcenter{\hbox{$#2#3$ }}\kern-.6\wd0}}

\newcolumntype{x}[1]{>{\centering\hspace{0pt}}p{#1}} 

\begin{document}
\begin{frontmatter}
\title{Stable cheapest nonconforming finite elements 
for the Stokes equations}
\tnotetext[label1]{The project is supported in part by
National Research Foundation of Korea (NRF--2013-0000153). 
}
\author[author1]{Sihwan Kim}\ead{sihwankim85@gmail.com}
\author[author2]{Jaeryun Yim}\ead{jaeryun.yim@gmail.com}
\author[author3]{Dongwoo Sheen\corref{cor1}}\ead{dongwoosheen@gmail.com}
\cortext[cor1]{Corresponding author.}
\address[author1]{Samsung Electronics Co., Giheung Campus, Nongseo-dong,
  Giheung-gu, Yongin-si, Gyeonggi-do 446--711, Korea}
\address[author2]{Interdisciplinary Program in Computational Science \&
Technology, Seoul National University, Seoul 151--747, Korea}
\address[author3]{Department of Mathematics and Interdisciplinary Program in
Computational Science \& Technology, Seoul National University, Seoul 151--747, Korea}

\begin{abstract}
We introduce two pairs of stable cheapest nonconforming 
finite element space pairs to approximate the Stokes equations.
One pair has each component of its velocity field to be approximated by 
the $P_1$ nonconforming quadrilateral element while 
the pressure field is approximated by the piecewise constant function with
globally two-dimensional subspaces removed: one removed space is due to
the integral mean--zero property and the other space consists of 
global checker--board patterns. The other pair consists of the
velocity space as the $P_1$ nonconforming quadrilateral element
enriched by a globally one--dimensional macro bubble function space based on $DSSY$ (Douglas-Santos-Sheen-Ye) nonconforming 
finite element space; 
the pressure field is approximated by the piecewise constant function with mean--zero space eliminated.
We show that two element pairs satisfy the discrete inf-sup condition uniformly.
And we investigate the relationship between them.
Several numerical examples are shown to 
confirm the efficiency and reliability of the proposed methods.
\end{abstract}

\end{frontmatter}

\section{Introduction}\label{sec:intro}
In the simulation of incompressible, viscous fluid mechanics, 
the lowest-degree conforming element
${\bP_1}\times P_0$ or $\bQ_1\times P_0$ produces numerically unstable solutions in 
the approximation of the pressure variable \cite{gira-ravi}. In particular 
Boland and Nicolaides \cite{boland-bilin,boland-stable} fully investigate
for the pair $\bQ_1\times P_0$. The above simple pair does not satisfy the discrete 
inf-sup condition. Several successful finite elements satisfying this condition have been
proposed and used. For instance conforming finite element spaces \cite{bernardi-raguel,
 engelman-sani-consistent-stokes, stenberg-stokes-unified, stenberg-stokes-some-problem}
including the $\bP_2\times P_0$ and $\bP_2\times P_1$ (the Taylor-Hood element) elements 
\cite{glowinski-pironneau-stokes,hood-taylor} and the MINI element \cite{mini-element}.

Instead of conforming finite element spaces, the use of nonconforming finite
element spaces has been regarded as one of the simplest resolutions to the
discrete inf-sup conditions: see \cite{cr73} for simplicial elements
with the $P_1$ nonconforming element for the velocity approximation and
the $P_0$ element for the pressure approximation.
For rectangular and quadrilateral elements, 
the use of nonconforming elements with four or five degrees of freedom
with the pressure approximation by $P_0$ element leads to stable element pairs for the Stokes equations 
\cite{han84, rann, dssy-nc-ell, cdy99,liu-yan-super, hu-man-shi, PSS-subspace,
  jeon-nam-sheen-shim-nonpara}.

The use of $P_1$ nonconforming quadrilateral element, whose local degrees of freedom
are only 3, in the approximation of velocity fields
with $P_0$ approximation to the pressure leads to unstable finite element
spaces.
An interesting question arises: what are the smallest rectangular/quadrilateral 
nonconforming element spaces to approximately solve the  velocity fields combined
with $P_0$ approximation to the pressure?

Recently, Nam {\it et al.} \cite{cheapest-nc-rect} introduced a cheapest
rectangular element based on the $P_1$ nonconforming quadrilateral element
\cite{parksheen-p1quad}
by adding a globally one-dimensional bubble function space \cite{rann, dssy-nc-ell} to the
${\bP_1}\times P_0$ pair on rectangular meshes. They show that the
one-dimensional enhancement to the velocity space fulfills the discrete
inf-sup
condition whose constant depends on the mesh size $h$ and provide several
convincing numerical results with smooth forcing term. However, it has been
questionable whether this one-dimensional modification can lead to a stable
cheapest element or not. 

The primary aim of this paper is to propose two stable cheapest finite element pairs based on 
the $P_1$ nonconforming quadrilateral element space and the piecewise constant
element space. Our modification is still a globally one--dimensional
enhancement to the velocity space
enriched by adding a globally one--dimensional $DSSY$-type (or Rannacher-Turek
type) bubble space based on macro interior edges.
Equivalently we propose to modify the pressure
space by eliminating a globally one--dimensional spurious mode with the
velocity
space unchanged from the $P_1$ nonconforming quadrilateral element space
(For a conforming counterpart, see \cite{gira-ravi}).
%

Indeed, these two finite element pairs are closely related.
We show that the velocity solutions obtained
by these two finite element pairs are identical 
while the pressure solutions differ only by a term $\Oh(h)$ times the global
discrete checker--board pattern. Thus, the stability and optimal convergence 
results for one finite element pair are equivalent to those for the other.

It should be stressed that if the conforming bilinear element is used instead
of our $P_1$ nonconforming quadrilateral element with the same modification to
the pressure space, the conforming bilinear element is still
not stable (See Cor. 5.1 and numerical results in \tabrefs{tab:4}{tab:5} in \secref{sec:numer}.

Recently, the proposed elements are used to solve a driven cavity
problem \cite{lim-sheen-cavity} and an interface problem governed by
the Stokes, Darcy, and Brinkman equations \cite{kim-sheen-brinkman}.

The outline of this paper is organized as follows. 
In Section 2, the Stokes problem will be stated and the first finite element pair will be defined.
In Section 3, we define the second finite element pair and present a
relationship between our two finite element pairs.
Section 4 will be devoted to check the discrete inf-sup condition 
for our proposed finite element pairs by using a technique derived
by Qin \cite{Qin}.
Finally, some numerical results are presented in Section 5.

\section{The Stokes problem and the stabilization of pressure space}\label{sec:stab-pressure}

In this section we will introduce a stable nonconforming finite
element space pair for the incompressible Stokes problem in two dimensions.
 We begin by examining the pair of
$P_1$ nonconforming quadrilateral element and the piecewise constant
element. Then a suitable minimal modification will be made so that uniform 
discrete inf-sup condition holds.

\subsection{Notation and preliminaries}
Let $\O \subset \mathbb R^2$ be a bounded domain with a polygonal boundary and
consider the following stationary Stokes problem:
\begin{subeqnarray}\label{eq:stokes}
  -\nu\De \bu+\grad p&=& {\bf f}\quad \mbox{in}\  \O, \slabel{eq:stokes1}\\
  \div\bu &=& 0 \quad \mbox{in}\  \O, \slabel{eq:stokes2}\\
  \bu &=& {\bf 0} \quad \mbox{on} \ \p\O,\slabel{eq:stokes3}
\end{subeqnarray}
where $\bu=(u_1,u_2)^T$ represents the velocity
vector, $p$ the pressure, ${\bf f}=(f_1, f_2)^T\in \bH^{-1}(\O)$ the body force,
and $\nu> 0$ the viscosity.
Set
$$
L_0^2(\O)=\displaystyle\{q\in L^2(\O) ~|~ \int_{\O} q~d{\bf x}=0\}.
$$
Here, and in what follows, we use the standard notations and definitions for the Sobolev
spaces $\bH^s(S)$, and their associated inner products
$(\cdot,\cdot)_{s,S}$, norms $||\cdot||_{s,S}$, and semi-norms
$|\cdot|_{s,S}.$ We will omit the subscripts $s, S$ if $s=0$ and $S=\O.$
Also for boundary $\p S$ of $S$, the inner product in $L^2(\p S)$ is denoted by
$\<\cdot,\cdot\>_S$.
Then, the weak formulation of \eqref{eq:stokes} is to seek 
a pair $(\bu,p)\in \bH_0^1(\O)\times L_0^2(\O) $ such that
\begin{subeqnarray}\label{eq:wform}
a(\bu,\bv) - b(\bv,p)&=& (\mathbf{f},\bv) \quad \forall \bv\in
\bH_0^1(\O), \slabel{eq:wforma} \\
  b(\bu,q) &=& 0 \quad\quad\quad \forall q\in L_0^2(\O), \slabel{eq:wformb}
\end{subeqnarray}
where the bilinear forms $a(\cdot,\cdot): \bH_0^1(\O)\times
\bH_0^1(\O)\rightarrow \mathbb R$ and $b(\cdot,\cdot): \bH_0^1(\O)\times
L_0^2(\O)\rightarrow \mathbb R$ are defined by
\begin{eqnarray*}
a(\bu,\bv)=\nu(\grad \bu, \grad \bv), \quad b(\bv,q)= (\div\bv,q).
\end{eqnarray*}
Let $\cdfs = \{\bv \in \bH_0^1(\O)~|~\div \bv = 0 \}$ denote the divergence--free
subspace of $\bH_0^1(\O)$. Then the solution $\bu$ of \eqref{eq:wform} lies in 
$\cdfs$ and satisfies
\begin{eqnarray}\label{eq:div-wform}
a(\bu,\bv) = (\mathbf{f},\bv) \quad \forall \bv \in \cdfs.
\end{eqnarray}

\subsection{Nonconforming finite element spaces}
In order to highlight our approach to design new finite element spaces,
we shall restrict our attention to the case of $\O=(0,1)^2.$
Let $(\Tau_h)_{0<h<1}$ be a family of uniform triangulation of $\O$ into
disjoint squares $Q_{jk}$ of size $h$ for $j,k=1,\cdots,N$
and ${\overline\O}=\bigcup^{N}_{j,k=1}{\overline Q}_{jk}$. ${\mathcal
  E}_h$ denotes the set of all edges in $\Tau_h$. Let $N_Q$ and $N_v^i$ be the number of
elements and interior vertices, respectively.
Let $P_j(Q)$ denote the space of polynomials of 
degree less than or equal to $j$ on region $Q$.

The approximate space for velocity fields is based on the 
$P_1$ nonconforming quadrilateral element \cite{cdssy, dssy-nc-ell,parksheen-p1quad}. 
Set
\begin{subeqnarray*}
&& {\vNCh} = \{\bv\in \bL^2(\O)\,\mid~ \bv|_{Q}\in \bP_1(Q) ~ \forall Q \in \Tau_h, \bv \text{ is continuous at the midpoint}~ \\ &&\qquad\qquad  \text{of each interior edge in } \Tau_h
\},
\end{subeqnarray*}
and
\begin{subeqnarray*}
&& {\vNChz} = \{\bv\in \vNCh \,\mid~ \bv \text{ vanishes at the midpoint of each boundary edge in } \Tau_h \}.
\end{subeqnarray*}
The pressure will be approximated by the space of piecewise constant functions with zero mean
$\Phz$, {\em i.e.},
\begin{eqnarray*}
\Phz = \{ q \in L^2_0(\O)\,\mid~ q|_{Q}\in P_0(Q) ~ \forall Q \in \Tau_h \},
\quad \dim(\Phz) = N_Q-1.
\end{eqnarray*}

It is known that the pair of spaces $\vNChz\times \Phz$ cannot be used to solve the Stokes equations,
as stated in the following theorem:

\begin{theorem}[\cite{cheapest-nc-rect}]\label{thm:singular} 
Let $(\Tau_h)_{0<h<1}$ be a family of 
triangulations of $\O$ into rectangles and set
\[
\gcbs=\{p_h\in \Phz  \,\mid  ~b_h(\bv_h,p_h)=0\quad \forall\bv_h\in \vNChz \},
\]
where $b_h(\bv_h,p_h):=\sum_{j=1}^{N_Q} (\nabla \cdot \bv_h,p_h)_{Q_j}$.
Then $\dim(\gcbs) = 1.$ Indeed, the elements $p_h\in \gcbs$ are
of global checker--board pattern.
\end{theorem}
Denote by $\gcb$ a global checker--board pattern basis function with $
\|\gcb\|=1$ such that
\begin{eqnarray}\label{eq:gcbs}
\gcbs = \Span\left\{\gcb\right\}.
\end{eqnarray}

For simplicity, we assume that $\Tau_h$ can be considered as the
disjoint union of macro elements such that each macro element consists
of $2 \times 2$ elements in $\Tau_h$.
For odd integers $j$ and $k$, consider the macro element 
$Q_{JK}^{M}$
consisting of $Q_{jk},$ $Q_{j,k+1},$ $Q_{j+1,k},$ and $Q_{j+1,k+1},$
with $(J,K) =(j,k).$ 
Denote by $\Tau^M$ the macro triangulation composed of all such macro elements $Q_{JK}$'s.
Let $p_{JK}^{mc} \in \Phz$ be the elementary checker--board pattern defined by
\begin{eqnarray*}
p_{JK}^{mc} = \begin{cases} \begin{bmatrix}
-1 &  1  \\
 1 & -1 
\end{bmatrix}
&\quad\text{on } Q_{JK}^{M}= \begin{bmatrix}
Q_{j,k+1}   & Q_{j+1,k+1}  \\
Q_{j,k}     & Q_{j+1,k}
\end{bmatrix}, \\
  0 &\quad\text{on }\O\setminus Q_{JK}^{M}.
\end{cases}
\end{eqnarray*}

\begin{figure}[htb!]\begin{center}
\epsfig{figure=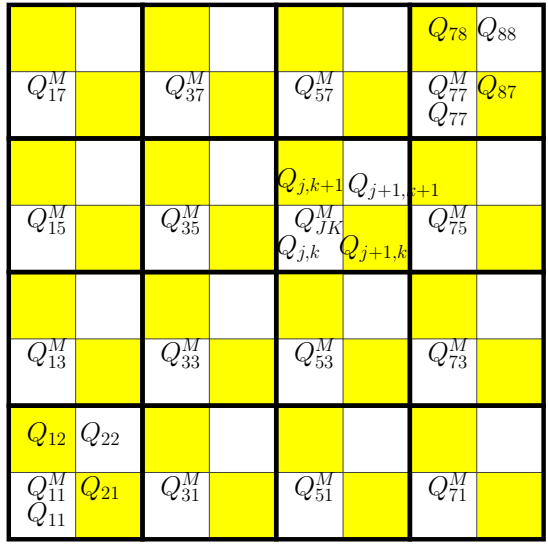,width=.60\textwidth}
\caption{Macro elements: $Q_{JK}^M = Q_{j,k}\cup Q_{j,k+1}\cup Q_{j+1,k}\cup
  Q_{j+1,k+1}, (J,K)=(j,k)$} \label{fig:bub_index}
\end{center}\end{figure}
We will employ capital letters to indicate odd integer indices for those macro
patterns on the macro element.
Owing to \thmref{thm:singular}, the global checker--board pattern
basis function $\gcb$ in \eqref{eq:gcbs} can be expressed explicitly as follows:
\begin{eqnarray}\label{eq:gcb}
\gcb = \sum_{JK} p_{JK}^{mc}.
\end{eqnarray}

We now try to stabilize $\vNChz \times \Phz$ minimally so that the modified pairs fulfill the
uniform inf-sup condition.  In this section we introduce
the stabilization of pressure approximation space $\Phz$ by eliminating one--dimensional global checker--board patterns from $\Phz.$
Alternatively, the stabilization of velocity approximation space
$\vNChz$, again with a globally one--dimensional modification, is given in \secref{sec:stab-velocity}.

\subsection{Stabilization of $\Phz$}
Define $\Pcf$ as the $L^2(\O)$--orthogonal complement of $\gcbs$ in $\Phz$, that is,
\begin{eqnarray} 
 \Phz =  \gcbs \oplus  \Pcf ,   \quad \dim(\Pcf)= N_Q -2.
\end{eqnarray} 
We are now ready to propose our Stokes element pair as follows:
\begin{eqnarray}
\vNChz\times \Pcf,\quad
\dim(\vNChz\times \Pcf) =2N_v^i+N_Q-2.
\end{eqnarray}

\subsection{The discrete Stokes problem}
Now define the discrete weak formulation of \eqref{eq:wform} to find a
pair $(\bu_h,p_h)\in \vNChz \times \Pcf$ such that
\begin{subeqnarray}\label{eq:wformd}
  a_h(\bu_h, \bv_h)-b_h(\bv_h,p_h) &=& ({\bf f},\bv_h)\quad \forall\bv_h\in
  \vNChz,
\slabel{eq:wformda} \\
  b_h(\bu_h,q_h)&=& 0 \quad\qquad~\forall q_h\in \Pcf,
\slabel{eq:wformdb}
\end{subeqnarray}
where the discrete bilinear forms $a_h(\cdot,\cdot): \vNChz\times \vNChz
\rightarrow \mathbb R$ and $b_h(\cdot,\cdot):\vNChz\times \Pcf \rightarrow
\mathbb R$
are defined in the standard fashion:
\begin{equation*}
a_h(\bu,\bv)=\nu \sum_{j=1}^{N_Q}(\grad \bu,\grad \bv)_{Q_j}\quad
\mbox{and}\quad b_h(\bv,q)= \sum_{j=1}^{N_Q}(\nabla\cdot\bv,q)_{Q_j}.
\end{equation*}
As usual, let $|\cdot|_{1,h}$ denote the (broken) energy semi-norm given by
\bes
|\bv|_{1,h} = \sqrt{a_h(\bv,\bv)},
\ees
which is equivalent to $\|\cdot\|_{1,h}$  on $\vNChz.$
Also, denote by $\|\cdot\|_{m,h}$ and $|\cdot|_{m,h}$ the usual
mesh-dependent norm and semi-norm:
\begin{equation*}
\|\bv\|_{m,h}=\bigg[\sum_{Q\in\mathcal{T}_h}\|\bv\|^2_{H^m(Q)}\bigg]^{1/2}
\quad\mbox{and}\quad 
|\bv|_{m,h}=\bigg[\sum_{Q\in\mathcal{T}_h}|\bv|^2_{H^m(Q)}\bigg]^{1/2},
\end{equation*}
respectively. 
Let $\gdfs$ denote the divergence--free subspace of $\vNChz$ to
$\Pcf$, {\it i.e.},
\be\label{eq:div-free}
\gdfs = \{\bv_h \in \vNChz~|~b_h(\bv_h,q_h) = 0,~\forall q_h \in \Pcf \}.
\ee
Then the solution $\bu_h$ of \eqref{eq:wformd} lies in 
$\gdfs$ and satisfies
\begin{eqnarray}\label{eq:div-wformd}
a_h(\bu_h,\bv_h) = (\mathbf{f},\bv_h)\quad \forall\bv_h \in \gdfs.
\end{eqnarray}
We state the main theorem of the paper, whose proof will be given in
\secref{sec:proof}.

\begin{theorem}\label{thm:p1p0-infsup}
$\vNChz\times\Pcf$ satisfies the uniform discrete inf-sup condition: 
\begin{eqnarray}\label{eq:infsup-a0}
\sup_{{\bv_{h}}\in  \vNChz}\frac{b_h({\bv_{h}},q_h)}{|{\bv_{h}}|_{1,h}} \geq 
\beta\|q_h\|_{0,\O} \quad\forall q_h\in \Pcf.
\end{eqnarray}
\end{theorem}

\section{Alternative stabilization by enriching the velocity space $\vNChz$}\label{sec:stab-velocity}

In this section we consider an enrichment of $\vNChz$ by 
adding a global one-dimensional bubble function space
based on the quadrilateral nonconforming bubble function 
\cite{cdssy,cdy99,dssy-nc-ell,jeon-nam-sheen-shim-nonpara}.
We then compare two proposed nonconforming finite element space pairs 
$\vNChz\times \Pcf$ and $\ttvNChz\times \Phz.$ Indeed, these two
spaces very closely related. The velocity solutions obtained by these two
spaces are identical while 
the difference between the two pressures isof order $\Oh(h)$.

On a reference domain $\hQ:=[-1,1]^2$, the $DSSY$ nonconforming element space
is defined by
$$
\DSSY(\hQ)=\Span\{1,\hat{x},\hat{y}, \theta_k(\hat{x})-\theta_k(\hat{y})\},
$$
where
$$
\theta_k(t) = \begin{cases}  t^2-\frac{5}{3}t^4,\quad   &k  = 1, \\
 t^2-\frac{25}{6}t^4+\frac{7}{2}t^6 ,\quad   &k  = 2.
\end{cases}
$$
Let $F_Q:\hQ \rightarrow Q$ be a bijective affine transformation from the
reference domain onto a rectangle $Q$. Then define
\be
\DSSY(Q) = \left\{ \hat{v}~\circ~F_{Q}^{-1}~\middle|~\hat{v}
\in \DSSY(\hat{Q}) \right\}.
\ee
The main characteristic of $\DSSY(Q)$ is the edge-mean-value property:
\be
\oint_{E} \bpsi~ds = \bpsi(\text{midpoint of E}) \quad \forall\bpsi \in \DSSY(Q),
\ee
where $\oint_{E}$ denotes $\frac{1}{|E|}\int_E$ \cite{dssy-nc-ell,jeon-nam-sheen-shim-nonpara}.

The vector-valued $DSSY$ nonconforming finite element space is defined by
\begin{subeqnarray*}
&&{\vDSSYhz}=\{\bv\in \bL^2(\O)~|~ \bv_{j}:=\bv|_{Q_{j}}\in
                 \vDSSY(Q_{j}) ~ \forall j=1,\cdots,N_Q; \\
&&\qquad\qquad  \qquad  \bv \text{ is continuous at the midpoint of each interior edge }  \\
&&\qquad\qquad  \qquad ~~ \text{ and vanishes at the midpoint of each boundary edge in } \Tau_h \}.
\end{subeqnarray*}
For each macro element $Q_{JK}^M,$  define $\bpsi_{Q_{JK}^M} \in \vDSSYhz$
such that
$
\supp(\bpsi_{Q_{JK}^M}) \subset \overline{Q}_{JK}^M,
$
and its integral averages over the edges in $\Tau_h$ vanish except on
the two edges $\p Q_{j,\ell} \cap \p Q_{j+1,\ell}$, $\ell=k,k+1:$ 
\bes
\oint_{\p Q_{j,k} \cap \p Q_{j+1,k}} \bpsi_{Q_{JK}^M}~ds = \bnu,\qquad
\oint_{\p Q_{j,k+1} \cap \p Q_{j+1,k+1}} \bpsi_{Q_{JK}^M}~ds = -\bnu.
\ees
where $\bnu$ denotes the unit outward normal vector of $Q_{j,\ell}$ on the edge
$\p Q_{j,\ell} \cap \p Q_{j+1,\ell}$, $\ell=k,k+1$. 
Define a basis function for the global bubble function, as shown in \figref{fig:macro-bubble-pss},
and a space of global bubble functions as follows:
\begin{eqnarray}
\gbbs=\Span\left\{ \gbb \right\}, \quad
\gbb = \dsum_{Q_{JK}^M\in\Tau^M}\bpsi_{Q_{JK}^M}.
\end{eqnarray}

 \begin{figure}[htb!] \begin{center}
 \epsfig{figure=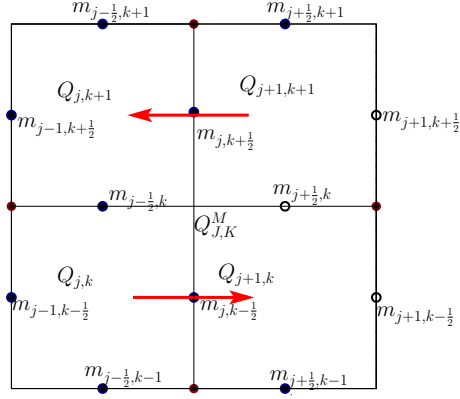,width=.50\textwidth}
 \caption{The basis function $\bpsi_{Q_{JK}^M}\in \vDSSYhz$, associated with the
   macro element $Q_{JK}^M$, 
 takes the value $\bnu$ and $-\bnu$ at the midpoints 
 $m_{j,k-\frac12}$ and $m_{j,k+\frac12},$ respectively, and value 0
 at any other midpoints $m$'s shown in the figure.
 $Q_{J,K}^M = Q_{j,k}\cup Q_{j,k+1}\cup Q_{j+1,k}\cup Q_{j+1,k+1}.$
 }\label{fig:macro-bubble-pss}
 \end{center} \end{figure}

We are now ready to enrich $\vNChz$ as follows:
\begin{equation}
\ttvNChz=\vNChz~\oplus~\gbbs.
\end{equation}

\begin{remark}
The dimension of the pair of spaces $\ttvNChz \times \Phz$ is $2N_v^i+N_Q$.
\end{remark}

We state the uniform inf-sup stability as in the following theorem,
whose
proof will be given in \secref{sec:proof}.
\begin{theorem}\label{thm:p1bp0-infsup}
$\ttvNChz\times \Phz$ satisfies the uniform discrete inf-sup condition.
\end{theorem}

\subsection{Comparison between $\vNChz\times \Pcf$ and $\ttvNChz\times \Phz$}
In this subsection, we will compare the two nonconforming finite element space pairs 
$\vNChz\times \Pcf$ and $\ttvNChz\times \Phz$.
These two pairs are closely related such that $\ttvNChz\times \Phz$ can be understood 
as a slight modification of $\vNChz\times \Pcf$.

For $\ttvNChz\times \Phz$, we have the following discrete weak
formulation:
{\em Find a pair $(\bu'_h,p'_h)\in \ttvNChz \times \Phz$ such that}
\begin{subeqnarray}\label{eq:wformd-2}
  a_h(\bu'_h, \bv'_h)-b_h(\bv'_h,p'_h) &=& ({\bf f},\bv'_h)\quad \forall\bv'_h\in
  \ttvNChz,
\slabel{eq:wformda-2} \\
  b_h(\bu'_h,q'_h)&=& 0 \quad\qquad~\forall q'_h\in \Phz.
\slabel{eq:wformdb-2}
\end{subeqnarray}
Let $\ttgdfs$ denote the divergence--free subspace of $\ttvNChz$ to $\Phz$,
{\it i.e.},
\be\label{eq:div-free-2}
\ttgdfs = \{\bv'_h \in \ttvNChz~|~b_h(\bv'_h,q'_h) = 0,~\forall q'_h \in \Phz \}.
\ee
Then the solution $\bu'_h$ of \eqref{eq:wformd-2} lies in 
$\ttgdfs$ and satisfies
\begin{eqnarray}\label{eq:div-wformd-2}
a_h(\bu'_h,\bv'_h) = (\mathbf{f},\bv'_h)\quad\forall\bv'_h \in \ttgdfs.
\end{eqnarray}
The following lemma implies that the two divergence--free subspaces
defined in \eqref{eq:div-free} and 
\eqref{eq:div-free-2} are identical, that is, our two proposed 
nonconforming finite element space pairs $\vNChz\times \Pcf$ and $\ttvNChz\times
\Phz$
produce an identical solution for velocity.

\begin{lemma}\label{lem:equiv-divfree}
The spaces $\gdfs$ and $\ttgdfs$ defined by \eqref{eq:div-free}
and \eqref{eq:div-free-2}, respectively, are equal.
\end{lemma}

\begin{proof}
Let $\bv_h \in \gdfs$ be given. Since $q_h' \in \Span\{\Pcf \oplus \gcbs\}$
and by \thmref{thm:singular}, we get $b_h(\bv_h,q_h')=0$. This implies $\bv_h \in \ttgdfs$,
so $\gdfs \subset \ttgdfs$. 
It remains to prove $\ttgdfs \subset \gdfs$. 
Let $\bv_h'= \bw_h + \bb_h \in \ttgdfs$ be given, where $\bw_h \in \vNChz$ and 
$\bb_h \in \gbbs$. In particular, if we consider 
$q_h' \in \gcbs$, then $b_h(\bv_h',q_h')=0$ implies $\bb_h \equiv {\bf 0}$. 
 Therefore $\bv_h' \in \vNChz$ and $b_h(\bv_h',q_h) =0$ for any $q_h
 \in \Pcf$ since $\Pcf \subset \Phz$. Hence $\bv_h' \in \gdfs,$ which
 shows  $\ttgdfs \subset \gdfs$. 
 This completes the proof.
\end{proof}
Owing to \lemref{lem:equiv-divfree}, 
$\bu_h \equiv \bu_h',$ where $\bu_h$ and $\bu_h'$ are the solutions of
\eqref{eq:wformd} and \eqref{eq:wformd-2}, respectively. Moreover,
 the difference between the two pressure solutions obtained by
\eqref{eq:wformda} and \eqref{eq:wformda-2} fulfills
\bes
b_h(\bv_h,p_h'-p_h)=0,\qquad \forall \bv_h \in \vNChz.
\ees
By \thmref{thm:singular}, $p'_h-p_h \in \gcbs$, that is, $p'_h$ can be represented by
\bes
p'_h = p_h + \alpha \gcb, \qquad \alpha \in \mathbb{R}.
\ees
Taking $\bv_h'=\gbb \in \gbbs$ in \eqref{eq:wformda-2}, we obtain
\be\label{eq:div-gbb-gcb}
\alpha b_h(\gbb,\gcb) &=& a_h(\bu_h,\gbb) - ({\bf f},\gbb) - b_h(\gbb,p_h), \nn\\
&=& \nu\dsum_{j=1}^{N_Q}(\grad \bu_h, \grad \gbb)_{Q_j} - ({\bf f},\gbb) - 
b_h(\gbb,p_h), \nn\\
&=& \nu\dsum_{j=1}^{N_Q}(-\De \bu_h,\gbb)_{Q_j} + \nu\left< \frac{\partial \bu_h}{\partial \bn} ,\gbb \right>_{\p Q_j}- ({\bf f},\gbb) - b_h(\gbb,p_h).
\ee
Since the solution $\bu_h$ is a piecewise linear polynomial, that is, $\bu_h \in \vNChz$, 
the first term in \eqref{eq:div-gbb-gcb} is equal to zero. And we easily check that 
the second and last terms in \eqref{eq:div-gbb-gcb} turn out to vanish
by the characteristics of the space $\gbbs$. 
A simple calculus using the Divergence Theorem yields
\be\label{eq:gbb-gcb}
b_h(\gbb,\gcb) = \frac1{h}.
\ee
Invoking \eqref{eq:gbb-gcb}, one obtains
\be
\alpha = - \frac{({\bf f},\gbb)}{b_h(\gbb,\gcb)} = -h ({\bf f},\gbb).
\ee
Hence,  $p'_h-p_h = -h (f,\gbb) \gcb.$

We summarize the above result as follows:
\begin{theorem}\label{thm:equiv-two-pair} Let
  $(\bu_h,p_h)\in\vNChz\times \Pcf$ and 
$(\bu'_h,p'_h) \in \ttvNChz\times \Phz$ are the solutions of 
\eqref{eq:wformd} and \eqref{eq:wformd-2}, respectively. Then
\begin{eqnarray}\bu_h = \bu'_h\quad\text{and}\quad 
p_h-p'_h =-h (f,\gbb) \gcb.
\end{eqnarray}
\end{theorem}


\subsection{Interpolation operator and conference results}
We recall from \cite{parksheen-p1quad} that the global interpolation operator 
$\bPi_h : \bH^2(\O) \rightarrow \vNCh$ is defined through the local interpolation
operator
$\bPi_Q : \bH^2(Q) \rightarrow \vNCh(Q)$
such that
\bes
\bPi_h|_Q = \bPi_Q\quad\forall Q\in \Tau_h.
\ees
Here, $\bPi_Q$ is explicitly defined by
\be\label{eq:local-interpol}
{\bPi}_Q\bw (M_k)= \frac{\bw(V_{k-1}) + \bw(V_{k})}{2}\qquad\forall \bw \in \bH^2(\O),
\ee
where $V_{k-1}$ and $V_k$ are the two vertices of the edge $E_k$ with midpoint
$M_k$ of $Q$.

Define an interpolation operator 
$S_h : H^1(\O)\cap L^2_0(\O) \rightarrow \Pcf$ by
\bes
(S_hq, z) &=& (q,z) \qquad  \forall z \in \Pcf.
\ees
Since $\bPi_h$ and $S_h$ reproduce linear and constant functions on each element $Q_j \in \Tau_h$
and macro element $Q_{JK}^M$, respectively, the standard polynomial
approximation results imply that
\begin{subeqnarray}\label{eq:poly-app}
&&\|\bv - \bPi_h\bv\|_0 + h|\bv - \bPi_h\bv|_{1,h} + 
h^2|\bv - \bPi_h\bv|_{2,h} \slabel{eq:poly-app-1}\\
&&\qquad\qquad+ h^{1/2}|\bv - \bPi_h\bv|_{0,\p \O}
 \leq Ch^2\|\bv\|_2\quad \forall
\bv \in \bH^2(\O),  \nn \\
&&\|q - S_hq\|_{0,\O} \leq Ch\|q\|_1 \quad \forall q \in H^1(\O)\cap
L^2_0(\O). \slabel{eq:poly-app-2}
\end{subeqnarray}
Owing to \eqref{eq:poly-app}, a standard application of
\thmrefs{thm:p1p0-infsup}{thm:p1bp0-infsup},
and the second Strang lemma yields the following optimal error estimate:
\begin{theorem}\label{thm:err-estimate}
Assume that \eqref{eq:stokes} is $H^2(\O)$--regular.
Let $(\bu,p)$ and $(\bu_h,p_h)$ be the solutions of \eqref{eq:wform}
and \eqref{eq:wformd}  (or \eqref{eq:wformd-2})
respectively.  Then the following optimal-order error estimate holds:
\bes\label{eq:dual-l2err}
\|\bu - \bu_h\|_0 + h \left[|\bu - \bu_h|_{1,h} + \|p-p_h\|_0\right] \leq Ch^2(|\bu|_2 + \|p\|_1).
\ees
\end{theorem}
\begin{remark}
In the above theorem, after the result for $\vNChz\times \Pcf$ is shown,
the corresponding result for $\ttvNChz\times \Phz$ 
to \thmref{thm:err-estimate} can be obtained a combination of
\thmrefs{thm:equiv-two-pair}{thm:err-estimate}. The order of
two spaces $\vNChz\times \Pcf$ and $\ttvNChz\times \Phz$ can be of
course exchanged.
\end{remark}
\section{Proofs of \thmrefs{thm:p1p0-infsup}{thm:p1bp0-infsup}
}\label{sec:proof}
In this section we will show that $\vNChz\times\Pcf$
and $\ttvNChz \times \Phz$
satisfy the uniform discrete inf-sup condition.
For this, some useful results \cite{gira-ravi, Qin} 
will be used; in particular, \lemref{lem:sub-infsup}, a result of Qin \cite{Qin},
will be utilized.

%
%
%
Our proof starts with setting
\begin{eqnarray*}
\Pch=\left\{q_h \in \Phz ~\middle|~q_h=\sum_{JK}a_{JK} p_{JK}^{mc},~\sum_{JK}a_{JK}=0 \right\},\quad\dim(\Pch) = \frac14 N_Q -1. 
\end{eqnarray*}
Then denote by $W^h$ the $L^2(\O)$--orthogonal complement of $\Pch$ in $\Phz$
such that
\begin{eqnarray}\label{eq:pcf}
\Pcf =  W^h \oplus  \Pch,
\quad\dim(\Pcf) = N_Q -2\quad\text{and }\, \dim(W^h) = \frac34 N_Q -1.
\end{eqnarray}

Let $\bZ^h$ denote the discrete divergence--free subspace of $\vNChz$
to $\Pch$, that is,
\begin{eqnarray*} 
\bZ^h = 
\left\{ \bv_h \in \vNChz~\middle|~b_h(\bv_h,q_h) = 0~\forall q_h \in \Pch \right\}. 
\end{eqnarray*}
Considering the conforming bilinear element
\be\label{eq:q1}
\vcQhz = \left\{ \bv_{bh} \in \bH_0^1(\O)~\middle|~ \text{each
  component of }\bv_{bh}|_Q 
\mbox{ is bilinear }~\forall Q \in \Tau_h \right\},
\ee
and $\bZ_b^h$ denote the discrete divergence--free subspace of  $\vcQhz$
to $\Pch$, that is, 
\begin{eqnarray*}
\bZ_b^h = \left\{ \bv_{bh} \in \vcQhz ~\middle|~b_h(\bv_{bh},q_h) = 0~
\forall q_h \in \Pch \right\}.
\end{eqnarray*}
Denote by $\cE_{2h}$ and $\cE_{2h}^i$ the sets of all edges and
interior edges, respectively, in  $\Tau^{M}$.
Set $\msvNChz$ to be the subspace of $\vNChz$ defined by
\be\label{eq:sub_p1nc}
\msvNChz=\left\{\bv_h\in \vNChz~\middle|~\bv_h= \dsum_{\Gamma^M \in {\cE}_{2h}}
\begin{pmatrix}a_{\Gamma^M}  \\ b_{\Gamma^M}\end{pmatrix} \psi_{\Gamma^M}
,\quad \begin{pmatrix}a_{\Gamma^M}
  \\ b_{\Gamma^M} \end{pmatrix}\in \mathbb R^2
\right\},
\ee
where $\psi_{\Gamma^M} \in \NChz$ is the basis function associated with the midpoint of
the macro edge $\Gamma^M \in {\cE}^i_{2h}$ as described in detail in the
caption of \figref{fig:macro-psi-jk}. 
Notice that
$\dim(\msvNChz) =  N_v^i-1.$

\begin{figure}[htb!]\begin{center}
\epsfig{figure=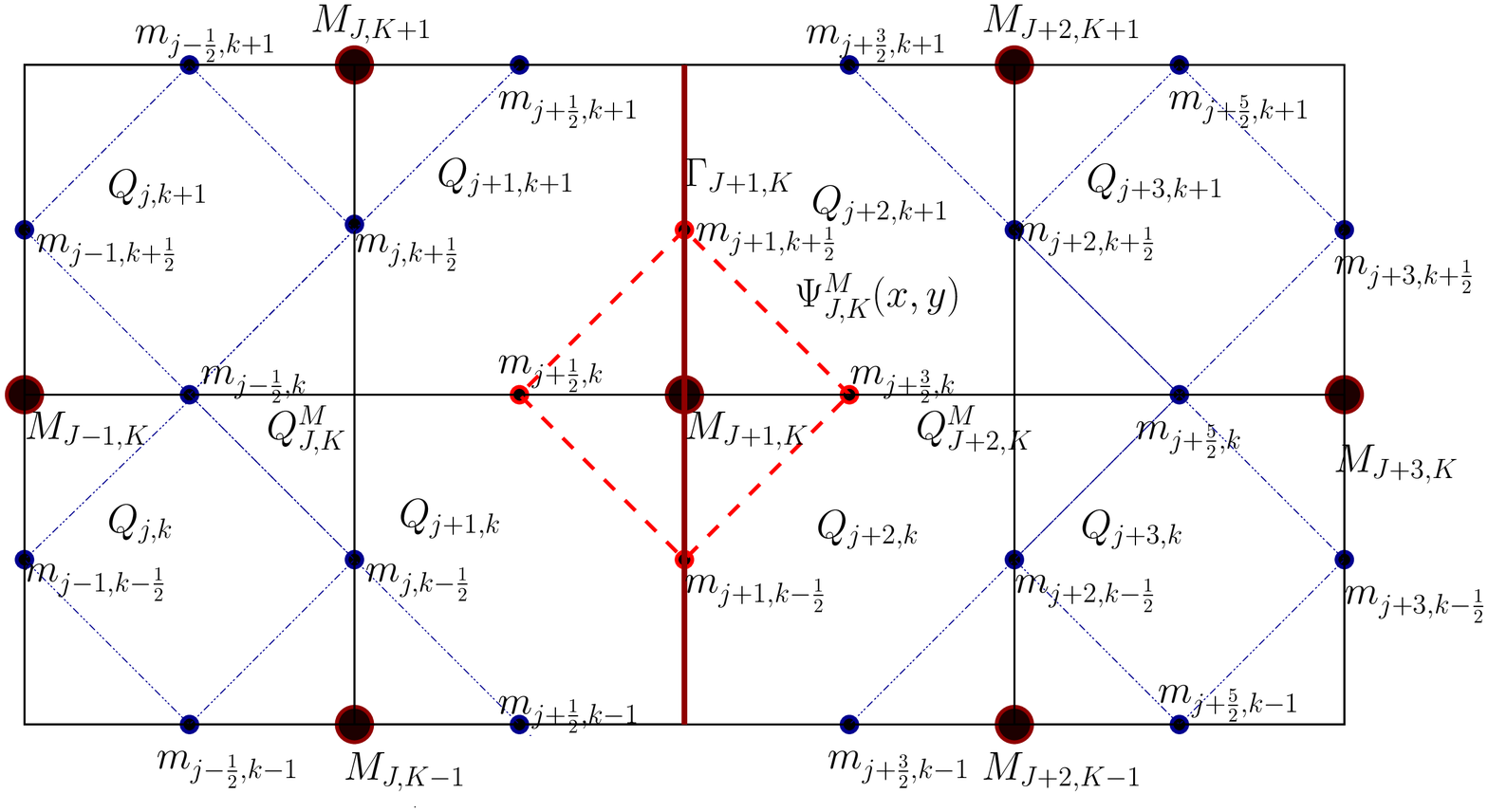,width=.95\textwidth}
\caption{The basis function $\psi_{\Gamma^M}\in\vNChz$, associated with the
  macro edge $\Gamma^M=\Gamma_{J+1,K}^M$, 
takes value 1 along the four line segments joining the midpoints 
$m_{j+\frac32,k}, m_{j+1,k+\frac12}, m_{j+\frac12,k},$ and $m_{j+1,k-\frac12}$, and value 0 
at any other midpoints $m$'s shown in the figure. $M_{J+1,K}$ denotes the midpoint of the macro edge 
$\Gamma_{J+1,K}^M$, the common edge of the two macro elements $Q_{J,K}^M$ and
$Q_{J+2,K}^M$, with
$Q_{J,K}^M = Q_{j,k}\cup Q_{j,k+1}\cup Q_{j+1,k}\cup Q_{j+1,k+1}$ and
$Q_{J+2,K}^M = Q_{j+2,k}\cup Q_{j+2,k+1}\cup Q_{j+3,k}\cup Q_{j+3,k+1}.$
}\label{fig:macro-psi-jk}
\end{center}\end{figure}

Next, we quote the Subspace Theorem of Qin as in the following lemma:
\begin{lemma}[\cite{Qin}]\label{lem:sub-infsup}
Given  $\bV^h \times P^h$,
let $\bV_1$ and $\bV_2$ be two subspaces of $\bV^h$ and $P_1$ and $P_2$ be two subspaces of  $P^h$.
Let the following four conditions hold:
\begin{enumerate}
\item[{\rm (1)}] $P^h$ $ = P_1 + P_2;$
\item[{\rm (2)}] there exist $\beta_j>0, j = 1, 2$, independent of $h$, such that
\bes
\sup_{{\bv_j}\in \bV_j}\frac{b_h({\bv_j},q_j)}{|{\bv_j}|_{1,h}} &\ge& 
\beta_j\|q_j\|_{0,\O}, \quad\forall q_j\in P_j,
\ees
\item[{\rm (3)}] there exist $\alpha_j \geq 0, j =1,2,$ such that
\bes
|b_h(\bv_j,q_k)| &\leq& \alpha_j |\bv_j|_{1,h}\|q_k\|_{0,\O},\quad \forall \bv_j \in \bV_j \mbox{ and } \forall
q_k \in P_k,\, j,k = 1,2; j\neq k,
\ees
with 
\bes
\alpha_1 \alpha_2 \le \beta_1 \beta_2.
\ees
\end{enumerate}
Then, $\bV^h \times P^h $ satisfies the inf-sup condition with the inf-sup constant depending
only on $\alpha_1,\alpha_2,\beta_1,\beta_2$.
\end{lemma}

\subsection{Proof of \thmref{thm:p1p0-infsup}}

The following lemma is an immediate consequence of the Divergence Theorem,
which will be useful to prove \lemref{lem:sub-infsup-1}:
\begin{lemma}\label{lem:div-interpol}
Let $Q \subset \mathbb R^2$ be a rectangular domain.
Suppose that $\bw$ is a two--variable function whose components are bilinear
polynomials on $Q$. Then the following holds:
\begin{eqnarray*}
\int_{Q} \div\bw~dA =\int_{Q} \div\bPi_{Q}\bw~dA.
\end{eqnarray*}
\end{lemma}

\begin{lemma}\label{lem:sub-infsup-1}
$\bZ^h\times W^h$ satisfies the uniform discrete inf-sup condition: 
\begin{eqnarray}\label{eq:infsup-a1}
\sup_{{\bv_{h}}\in  \bZ^h}\frac{b_h({\bv_{h}},q_h)}{|{\bv_{h}}|_{1,h}} \geq 
\beta\|q_h\|_{0,\O} \quad\forall q_h\in W^h.
\end{eqnarray}
\end{lemma}
\begin{proof}
We begin with invoking \cite{boland-stable} that $\bZ_b^h
\times W^h$ 
satisfies the uniform inf-sup condition, that is,
there exists a positive constant $\beta$ independent of $h$ such that
\begin{eqnarray}\label{eq:infsup-cq1-w}
\sup_{{\bv_{bh}}\in  \bZ_b^h}\frac{b_h({\bv_{bh}},q_h)}{|{\bv_{bh}}|_{1,h}} \geq 
\beta\|q_h\|_{0,\O}\quad\forall q_h\in W^h.
\end{eqnarray}
Let $q_h \in W^h, q_h\neq 0$ be arbitrary. Then, 
\eqref{eq:infsup-cq1-w} is equivalent (cf. \cite{gira-ravi}, p. 118) to the
existence of $\bv_{bh} \in \bZ_b^h$ such that
\begin{subeqnarray}\label{eq:infsup-cq1-w-equiv}
b_h(\bv_{bh},q_h) = \|q_h\|_{0,\O}^2, \slabel{eq:infsup-cq1-w-equiv-1}\\
|\bv_{bh}|_{1,\O} \leq \frac{1}{\beta}\|q_h\|_{0,\O}. \slabel{eq:infsup-cq1-w-equiv-2}
\end{subeqnarray}


Now \lemref{lem:div-interpol} implies that
$\bPi_h\bv_{bh} \in \bZ^h$ and
\begin{eqnarray}\label{eq:infsup-p1-w-equiv-1}
b_h(\bPi_h\bv_{bh},q_h) = b_h(\bv_{bh},q_h) = \|q_h\|_{0,\O}^2.
\end{eqnarray}
By Young's inequality, the definition of interpolation operator $\bPi_h$ and \eqref{eq:infsup-cq1-w-equiv-2},
one sees that
\be\label{eq:infsup-p1-w-equiv-2}
|\bPi_h\bv_{bh}|_{1,h} \leq C|\bv_{bh}|_{1,\O} \leq \frac{C}{\beta}\|q_h\|_{0,\O},
\ee
where the constant $C$ is independent of mesh size $h$.
Notice that the element of $\bv_h=\bPi_h\bv_{bh} \in \bZ^h$ satisfying
\eqref{eq:infsup-p1-w-equiv-1} and \eqref{eq:infsup-p1-w-equiv-2} plays a role of an equivalent statement to
\eqref{eq:infsup-a1}. Hence the lemma is complete.
\end{proof}

\begin{lemma}\label{lem:sub-infsup-2}
$\msvNChz\times \Pch$ satisfies the uniform discrete inf-sup condition:
\begin{eqnarray}\label{eq:infsup-a2}
\sup_{{\bv}_{h}\in  \msvNChz}\frac{b_h({\bv}_{h},q_h)}{|{\bv}_{h}|_{1,h}} \geq 
\beta\|q_h\|_{0,\O}\quad\forall q_h\in \Pch.
\end{eqnarray}
\end{lemma}
\begin{proof}
Set
\begin{eqnarray*}
\Phhz = \{ q \in L^2_0(\O)\,\mid~ q|_{Q^M}\in P_0(Q^M) ~ \forall Q^M \in \Tau^{M} \},
\quad \dim(\Phhz) = N_Q/4-1.
\end{eqnarray*}
Due to Lemma 3.1 in \cite{PSS-subspace}, $\msvNChz \times \Phhz$
satisfies the uniform inf-sup condition, that is,
there exists a positive constant $\beta$ independent of $h$ such that
\begin{eqnarray}\label{eq:infsup-p1m-p0m}
\sup_{\overline{\bv}_{h}\in  \msvNChz}\frac{b_h(\overline{\bv}_{h},\bq)}{|\overline{\bv}_{h}|_{1,h}} \geq 
\beta\|\bq\|_{0,\O}\quad\forall \bq\in \Phhz.
\end{eqnarray}
Let $q_h=\dsum_{JK}\alpha_{JK} p_{JK}^{mc} \in \Pch$ be arbitrary.
Consider
$\bq=\dsum_{JK}\alpha_{JK} p_{JK} \in \Phhz,$ where $p_{JK}=\chi_{Q_{JK}^M}.$
Then there exists $\overline{\bv}_h= \dsum_{\Gamma^M \in {\cE}_{2h}}
\begin{pmatrix}a_{\Gamma^M}  \\ b_{\Gamma^M}\end{pmatrix} \psi_{\Gamma^M}
\in \msvNChz$ such that \eqref{eq:infsup-p1m-p0m} holds.
From this $\overline{\bv}_h$, we define $\bv_h \in \msvNChz$ as follows:
\bes
\bv_h = -\dsum_{\Gamma^M \in {\cE}_{2h}}
\begin{pmatrix}b_{\Gamma^M}  \\ a_{\Gamma^M}\end{pmatrix} \psi_{\Gamma^M}.
\ees
Then the following three equalities are obvious:
\bs\label{eq:three-equal}
\|q_h\|_{0,\O} &=& \|\bq\|_{0,\O}, \\
|\bv_h|_{1,h} &=& |\overline{\bv}_h|_{1,h}, \\
b_h(\bv_h,q_h) &=& b_h(\overline{\bv}_h,\bq).
\es
From \eqref{eq:infsup-p1m-p0m} and \eqref{eq:three-equal}, the inf-sup condition
\eqref{eq:infsup-a2} for $\msvNChz\times \Pch$ follows. This proves our assertion.
\end{proof}

Utilizing \lemref{lem:sub-infsup}, we are ready to prove 
\thmref{thm:p1p0-infsup}.

\begin{proof}[Proof of \thmref{thm:p1p0-infsup}.]
We will check the conditions of \lemref{lem:sub-infsup}. 
Let $\bV_1=\bZ^h,~\bV_2=\msvNChz$ and $P_1=W^h,~P_2=\Pch$. Obviously, 
$\bV_j$ and $P_j,~j=1,2$ are subspaces of $\vNChz$ and $\Pcf$, respectively,
so that Condition (1) holds. Moreover,
\lemrefs{lem:sub-infsup-1}{lem:sub-infsup-2} imply that Condition (2) holds.
Since $b_h(\bv_1,q_2)=0$ holds for any $\bv_1 \in \bV_1$
and any $q_2 \in P_2$, one has $\alpha_1=0$. Consequently, Condition (3) holds. 
Hence by \lemref{lem:sub-infsup},  $\vNChz\times \Pcf$ 
satisfies the inf-sup condition \eqref{eq:infsup-a0}.
\end{proof}

\subsection{Proof of \thmref{thm:p1bp0-infsup}}
In order to prove \thmref{thm:p1bp0-infsup}, the following lemma is needed.
\begin{lemma}\label{lem:sub-infsup-3}
$\gbbs\times \gcbs$ satisfies the inf-sup condition, 
that is, there exists a positive constant $\beta$ independent of $h$ such that
\begin{eqnarray}\label{eq:infsup-a3}
\sup_{{\bv}_{h}\in  \gbbs}\frac{b_h({\bv}_{h},q_h)}{|{\bv}_{h}|_{1,h}} \geq 
\beta\|q_h\|_{0,\O}\quad\forall q_h\in \gcbs.
\end{eqnarray}
\end{lemma}
\begin{proof}
Let $q_h \in \gcbs$ be given by $q_h= \alpha \gcb$ with a constant
$\alpha \in \mathbb R,$ and set $\bv_h =\gbb \in \gbbs.$ Recall
\eqref{eq:gbb-gcb} so that
\be\label{eq:gbb-gcb-alpha}
b_h(\bv_h, q_h) = \frac{\alpha}{h}.
\ee
Also, it is trivial to see
\be\label{eq:gcb-alpha}
\|q_h\|_{0,\O} = |\alpha|.
\ee
It remains to compute $|\bv_h|_{1,h}$. For this, we notice that $|\bv_h|_{1,Q}$ does not depend on the mesh size $h$ of $Q$,
since it is a two dimensional region. Indeed, there exists a constant $C_1$ independent of $h$ such that
$
|\bv_h|_{1,h}^2 = \dsum_{Q\in\Tau_h}\int_{Q}|\grad \bv_h|^2~d{\bx} = \frac{C_1}{h^2}.
$
Hence, we get 
\be\label{eq:gbb}
|\bv_h|_{1,h} = \frac{C}{h}, \qquad\text{where }C=\sqrt{C_1}.
\ee
Now, the combination of \eqref{eq:gbb-gcb-alpha}, \eqref{eq:gcb-alpha} and \eqref{eq:gbb}
leads to
\eqref{eq:infsup-a3} with the inf-sup constant $\beta=1/C$. This completes the proof.
\end{proof}

\thmref{thm:p1bp0-infsup} is now ready to be shown, by
using \lemrefs{lem:sub-infsup}{lem:sub-infsup-3}.
\begin{proof}[Proof of \thmref{thm:p1bp0-infsup}.]
Let $\bV_1=\vNChz,~\bV_2=\gbbs$ and $P_1=\Pcf,~P_2=\gcbs$.
Since $\Phz= \Pcf \oplus \gcbs$, Condition $(1)$ in \lemref{lem:sub-infsup} holds. Moreover, \thmref{thm:p1p0-infsup} and
\lemref{lem:sub-infsup-3} imply Condition $(2)$ holds. Finally, 
$b_h(\bv_1,q_2)=0$ holds for any $\bv_1 \in \bV_1$
and $q_2 \in P_2$ by \thmref{thm:singular}, which implies that $\alpha_1=0.$ 
Consequently, Condition $(3)$ holds. 
Hence, $\ttvNChz\times \Phz$ satisfies the inf-sup condition.
Note that the constant in each step is independent of $h$.
\end{proof}

\section{Numerical results}\label{sec:numer}
Now we illustrate a numerical example for the stationary
Stokes problem on uniform meshes on the domain $\O=(0,1)^2.$
Throughout this numerical study, we fix $\nu=1$. 

First we calculate the discrete inf-sup constants of various finite element
pairs including our suggestions.

In contrast to the $\Oh(h)$--dependent inf-sup constant of conforming bilinear and piecewise
constant finite element pair \cite{boland-bilin,boland-stable},
our two proposed nonconforming finite elements satisfy the uniform inf-sup condition
at least on square meshes. To confirm theoretical analysis, we give the 
numerical  results of the discrete inf-sup constants \cite{malkus} in \tabref{tab:1}.
\begin{table}[ht]
  \begin{center}
    \begin{tabular}{r |c c|c c|c c}\hline
      $h$ & $\beta_1$ &Order& $\beta_2$ &Order& $\beta_3$ &Order \\
\hline\hline
      ${1}\slash{4}$ & 4.9642E-01 &  -   & 4.9560E-01 &  -   & 5.0000E-01 & -\\
      ${1}\slash{8}$ & 2.8605E-01 & 0.78 & 4.6791E-01 & 0.08 & 4.6746E-01 & 0.09\\
      ${1}\slash{16}$& 1.5029E-01 & 0.93 & 4.4415E-01 & 0.07 & 4.5296E-01 & 0.04\\
      ${1}\slash{32}$& 7.6544E-02 & 0.97 & 4.2863E-01 & 0.05 & 4.4526E-01 & 0.02\\
      ${1}\slash{64}$& 3.8562E-02 & 0.99 & 4.1864E-01 & 0.03 & 4.4051E-01 & 0.02\\
      \hline
    \end{tabular}
  \end{center}
  \caption{\label{tab:1} Estimation of the inf-sup constants 
$\beta_j,j=1,2,3,$ for the three finite element pairs 
$\vcQhz\times \Pcf$, $\vNChz\times \Pcf$, and $\ttvNChz\times \Phz.$}
\end{table}

We will borrow the two numerical examples
from \cite{PSS-subspace}. The source term ${\bf f}$ is generated by the choice
of the exact solution.
\be\label{eq:numer-test}
\bu(x,y) = (s(x)s'(y),-s(y)s'(x)), \qquad p(x,y) = \sin(2\pi x)f(y),
\ee
where $s(t) = \sin(2\pi t)(t^2-t)$ and $s'(t)$ denotes its derivative. 
The velocity $\bu$ vanishes on $\p \O$ and the
pressure $p$ has mean value zero regardless of $f$.

Several interesting numerical results for the pair $\vNChz\times\Pcf$ are
presented, while the corresponding numerical results for the pair
$\ttvNChz\times \Phz$ are omitted here, since they behave quite similarly 
to those case for the pair $\vNChz\times \Pcf$.
Numerical results with $f(y)=\frac{1}{3-\tan^2y}$ are shown in \tabref{tab:2}.  
We observe optimal order of convergence in both velocity and pressure variables.
Also numerical experiments are carried out and presented in \eqref{eq:numer-test} for
$f(y)=\frac{1}{25-10\tan^2y} + \frac{3}{10}$ which
has a huge slope near the boundary on $y=1$. 
Since the pressure changes rapidly on the boundary $y=1$, 
convergence rates show a poor approximation in coarse meshes in \tabref{tab:3}.
However, as the meshes get finer, optimal order convergence is observed as expected
from the inf-sup condition. 

\begin{table}[ht]
\centering
\begin{tabular}{ r | c  c | c  c | c  c }\hline

  h & $|\bu-\bu_h|_{1,h}$ & Order & $\|\bu-\bu_h\|_{0}$ & Order &
  $\|p-p_h\|_{0}$ & Order \\ \hline \hline
  1/4 & 1.5087E-0 & - & 2.1583E-1 & - & 2.2190E-1 & - \\
  1/8 & 8.1269E-1 & 0.8926 & 5.5033E-2 & 1.9715 & 1.4098E-1 & 0.6544 \\
  1/16 & 4.1360E-1 & 0.9745 & 1.3930E-2 & 1.9821 & 6.4738E-2 & 1.1229 \\
  1/32 & 2.0767E-1 & 0.9939 & 3.4936E-3 & 1.9954 & 3.2509E-2 & 0.9938 \\
  1/64 & 1.0394E-1 & 0.9985 & 8.7411E-4 & 1.9988 & 1.6411E-2 & 0.9862 \\
  1/128 & 5.1985E-2 & 0.9996 & 2.1857E-4 & 1.9997 & 8.2359E-3 & 0.9947 \\
  1/256 & 2.5994E-2 & 0.9999 & 5.4646E-5 & 1.9999 & 4.1222E-3 & 0.9985 \\
  1/512 & 1.2997E-2 & 1.0000 & 1.3661E-5 & 2.0000 & 2.0616E-3 & 0.9996 \\
  1/1024 & 6.4987E-3 & 1.0000 & 3.4154E-6 & 2.0000 & 1.0309E-3 & 0.9999 \\\hline
 
\end{tabular}
\caption{\label{tab:2} Numerical results for uniform meshes with 
$f(y) = \frac{1}{3-\tan^2 y}$}
\end{table}

\begin{table}[ht]
\centering
\begin{tabular}{ r | c  c | c  c | c  c }\hline
  h & $|\bu-\bu_h|_{1,h}$ & Order & $\|\bu-\bu_h\|_{0}$ & Order &
  $\|p-p_h\|_{0}$ & Order \\ \hline \hline
  1/4 & 1.5086E-0 & - & 2.1578E-1 & - & 1.7459E-1 & - \\
  1/8 & 8.1268E-1 & 0.8925 & 5.5016E-2 & 1.9716 & 1.1835E-1 & 0.5609 \\
  1/16 & 4.1360E-1 & 0.9744 & 1.3926E-2 & 1.9820 & 5.7158E-2 & 1.0501 \\
  1/32 & 2.0767E-1 & 0.9939 & 3.4938E-3 & 1.9950 & 3.6347E-2 & 0.6531 \\
  1/64 & 1.0394E-1 & 0.9985 & 8.7450E-4 & 1.9983 & 2.3178E-2 & 0.6491 \\
  1/128 & 5.1985E-2 & 0.9996 & 2.1872E-4 & 1.9993 & 1.3569E-2 & 0.7725 \\
  1/256 & 2.5994E-2 & 0.9999 & 5.4690E-5 & 1.9998 & 7.3091E-3 & 0.8925 \\
  1/512 & 1.2997E-2 & 1.0000 & 1.3673E-5 & 1.9999 & 3.7516E-3 & 0.9622 \\
  1/1024 & 6.4987E-3 & 1.0000 & 3.4183E-6 & 2.0000 & 1.8899E-3 & 0.9892 \\\hline
\end{tabular}
\caption{\label{tab:3}Numerical results for uniform meshes with 
$f(y) = \frac{1}{25-10\tan^2 y} + \frac{3}{10}$}
\end{table}

The following numerical results highlight the reliability of
our proposed finite element space compared to the case of using the conforming
bilinear element for the approximation of the velocity field.
Recall that the pair of conforming finite element space combined with the
piecewise constant element space $\vcQhz\times\Pcf$ is unstable unless $\mathbf f$ 
is smooth enough as quoted in the following Corollary:
\begin{corollary} [Boland and Nicolaides, Cor. 6.1 in \cite{boland-stable}]\label{cor:slow-conv}
For $\beta\in(0,1),$ there exists $\mathbf f\in \bL^2(\O)$ such that
the pressure approximation to \eqref{eq:wform} by using
$\vcQhz\times \Pcf$ fulfills
\be\label{eq:slow-conv}
\|p-p_h\|_0 \geq Ch^{\beta}\|\mathbf f\|_0\qquad \text{for }  h\le h_\beta
\ee
for some $h_\beta>0,$ independent of $h.$
\end{corollary}
With $\beta=0.3$ fixed, some comparative numerical results for conforming and nonconforming
pairs using  $\vcQhz\times\Pcf$ and  $\vNChz\times\Pcf$ are shown
in \tabrefs{tab:4}{tab:5}, respectively. These results ensure 
the superiority of our nonconforming method over the conforming counterpart.

Throughout our numerical experiments, the $4\times 4$ Gauss quadrature
rule is adopted for each rectangular element. The approximate data for $\bf f$ 
are calculated by following the proof of Theorem 6.1 in \cite{boland-stable}
at the $4\times 4$ Gauss points in each element of $512 \times 512$ mesh.
The reference solutions used in error calculation are
obtained by using the $DSSY$ element \cite{dssy-nc-ell} with the $512\times 512$ mesh.
The graphs of components of $\mathbf f$ are given in \figref{fig:slow-data}.

\begin{remark}
It should be stressed that the degrees of freedom for both
$\vcQhz\times \Pcf$ and $\vNChz\times \Pcf$ are essentially identical,   
although numerical results are quite different. Further investigations 
need to be sought to analyze the differences between the conforming bilinear element
and the $P_1$ nonconforming element.
\end{remark}

\begin{figure}[ht]
\subfigure
{\epsfig{figure=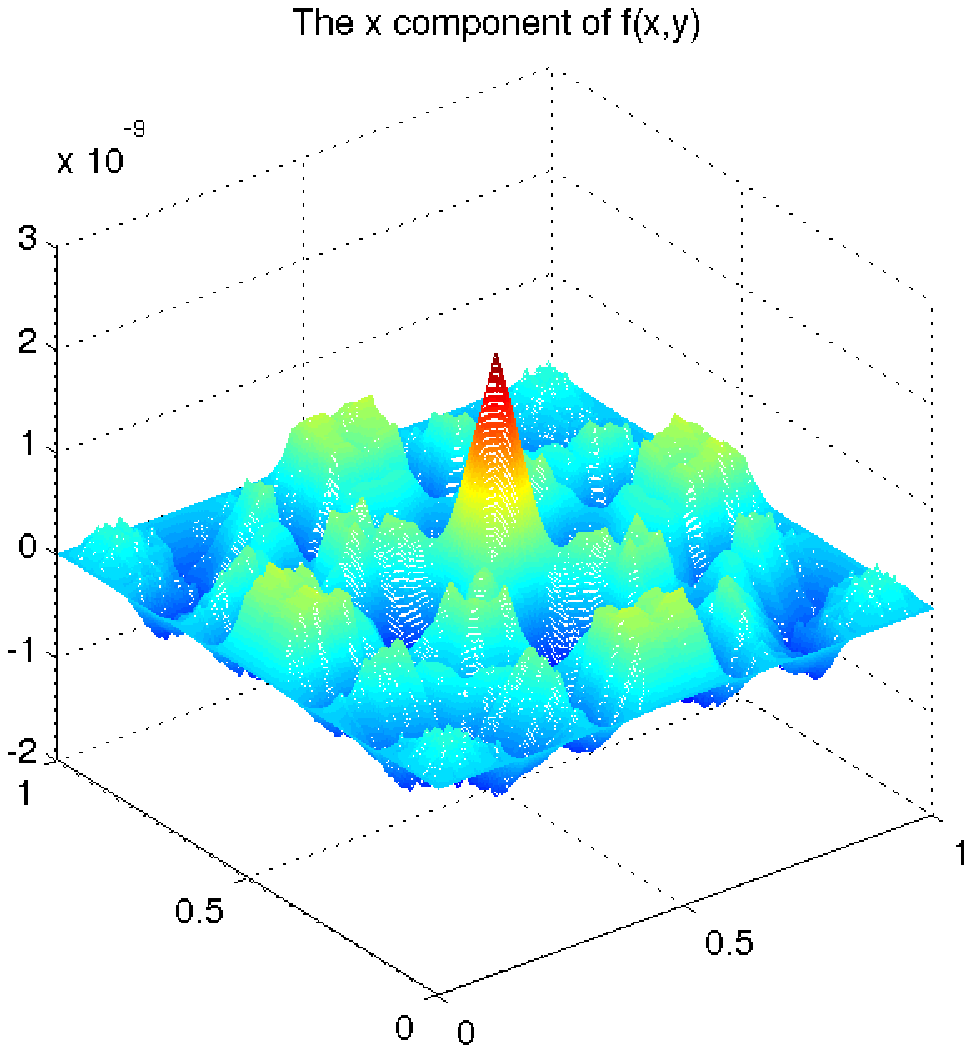,width=.40\textwidth}}\quad
\hspace{1cm}
\subfigure
{\epsfig{figure=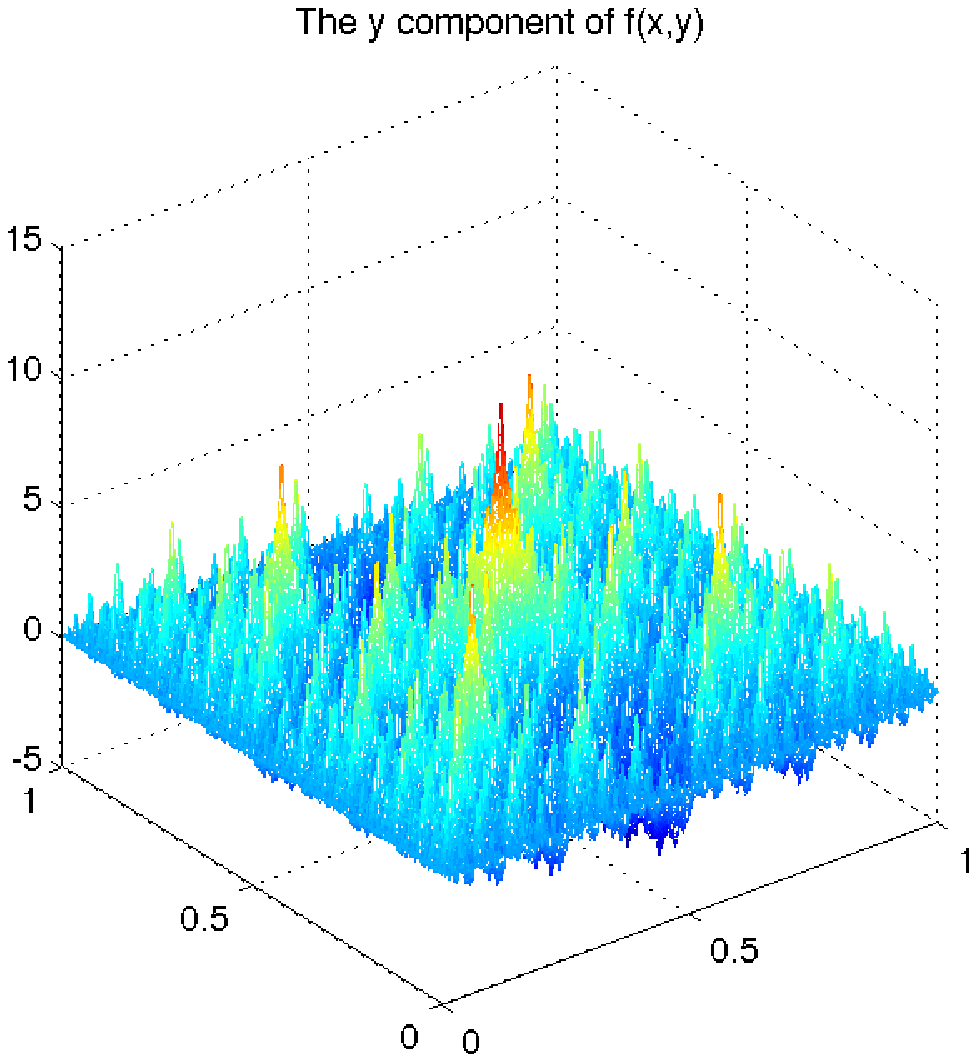,width=.40\textwidth}}
\caption{The graph of data $\bf f$}\label{fig:slow-data}
\end{figure}

\begin{table}[ht]
\centering
\begin{tabular}{r | c  c | c  c | c  c }\hline
  h & $|\bu_{ref}-\bu_h|_{1,h}$ & order & $\|\bu_{ref}-\bu_h\|_0$ & order & $\|p_{ref}-p_h\|_0$ & order \\ \hline\hline
  1/4 & 2.8248E-2 & - & 1.8470E-3 & - & 7.2967E-2 & - \\
  1/8 & 1.6008E-2 & 0.8193 & 5.3114E-4 & 1.7981 & 5.6105E-2 & 0.3791 \\
  1/16 & 8.5909E-3 & 0.8980 & 1.4266E-4 & 1.8964 & 4.1920E-2 & 0.4205 \\
  1/32 & 4.4824E-3 & 0.9385 & 3.7531E-5 & 1.9265 & 3.1925E-2 & 0.3929 \\
  1/64 & 2.3084E-3 & 0.9573 & 9.6932E-6 & 1.9531 & 2.4932E-2 & 0.3567 \\
  1/128 & 1.1939E-3 & 0.9512 & 2.4703E-6 & 1.9722 & 1.9829E-2 & 0.3304 \\
  1/256 & 6.4542E-4 & 0.8874 & 6.2940E-7 & 1.9727 & 1.5938E-2 & 0.3152 \\\hline
\end{tabular}
\caption{Numerical results for $\vcQhz\times\Pcf$
when $\beta=0.3$}\label{tab:4}
\end{table}

\begin{table}[ht]
\centering
\begin{tabular}{r | c  c | c  c | c  c }\hline
  h & $|\bu_{ref}-\bu_h|_{1,h}$ & order & $\|\bu_{ref}-\bu_h\|_0$ & order & $\|p_{ref}-p_h\|_0$ & order \\ \hline\hline
  1/4 & 2.8359E-2 & - & 1.8561E-3 & - & 4.9406E-2 & - \\
  1/8 & 1.7966E-2 & 0.6585 & 5.0224E-4 & 1.8858 & 2.6963E-2 & 0.8737 \\
  1/16 & 1.0379E-2 & 0.7916 & 1.3390E-4 & 1.9072 & 1.4305E-2 & 0.9144 \\
  1/32 & 5.6226E-3 & 0.8844 & 3.5144E-5 & 1.9298 & 7.5726E-3 & 0.9177 \\
  1/64 & 2.9406E-3 & 0.9351 & 9.0617E-6 & 1.9554 & 3.9235E-3 & 0.9486 \\
  1/128 & 1.5002E-3 & 0.9710 & 2.3029E-6 & 1.9763 & 1.9663E-3 & 0.9966 \\
  1/256 & 7.3601E-4 & 1.0274 & 5.7096E-7 & 2.0120 & 8.9372E-4 & 1.1376 \\\hline
\end{tabular}
\caption{Numerical results for $\vNChz\times\Pcf$
when $\beta=0.3$}\label{tab:5}
\end{table}

\section*{Acknowledgments}
This research was partially supported by NRF of Korea (Nos. 2013-0000153).

\bibliographystyle{abbrv}

\end{document}